\documentclass[12pt]{article}

\usepackage[margin=2cm]{geometry}

\usepackage{amsmath}
\usepackage{amsthm}
\usepackage{amsfonts}
\usepackage{amssymb}
\usepackage{mathtools}
\usepackage{bbm}
\usepackage{enumitem}
\usepackage[colorlinks,citecolor=uniBlue,urlcolor=uniBlue,linkcolor=uniRed]{hyperref}

\usepackage{caption}
\usepackage{subcaption}

\usepackage{tikz}
\usetikzlibrary{arrows.meta}

\usepackage{xpatch}
\xpretocmd{\proof}{\setlength{\parindent}{0pt}}{}{}

\definecolor{uniRed}{rgb}{1,0.07,0.04}
\definecolor{uniBlue}{rgb}{0,0.67,0.86}
\definecolor{lGrey}{RGB}{220,220,220}

\usepackage{cleveref}
\theoremstyle{definition}
\newtheorem{defi}{Definition}

\theoremstyle{plain}

\newtheorem{thm}[defi]{Theorem}
\newtheorem*{thm*}{Theorem}

\newtheorem{lem}[defi]{Lemma}
\newtheorem*{lem*}{Lemma}

\newtheorem{prop}[defi]{Proposition}

\theoremstyle{remark}
\newtheorem{rem}[defi]{Remark}

\crefname{lem}{lemma}{lemmas}
\crefname{thm}{theorem}{theorems}

\DeclareMathOperator{\cov}{Cov}

\DeclareMathOperator{\dom}{dom }
\DeclareMathOperator{\oL}{L}

\DeclareMathOperator{\MalD}{D}

\DeclareMathOperator{\dist}{dist}
\DeclareMathOperator{\opHess}{Hess}

\newcommand{\opL}{\oL^{-1}}

\newcommand{\R}{\mathbb{R}}
\newcommand{\N}{\mathbb{N}}

\newcommand{\p}{\mathbb{P}}

\newcommand{\E}{\mathbb{E}}

\newcommand{\X}{\mathbb{X}}

\newcommand{\1}{\mathbbm{1}}

\newcommand{\lat}{L^{(\alpha)}_t}

\newcommand{\lati}[1]{L^{(\alpha_{#1})}_t}
\newcommand{\tillati}[1]{\tilde{L}^{(\alpha_{#1})}_t}

\newcommand{\DD}{\MalD^{(2)}}
\newcommand{\mal}{\MalD}

\newcommand{\kder}[2]{\norm{#1^{(#2)}}_\infty}

\renewcommand{\[}{\begin{equation}}
\renewcommand{\]}{\end{equation}}

\newcommand{\ind}[1]{\mathbbm{1}_{\{#1\}}}
\newcommand{\norm}[1]{\lVert #1 \rVert}

\numberwithin{equation}{section}

\title{Multivariate Second-Order $p$-Poincaré Inequalities}
\author{Tara Trauthwein\thanks{Department of Mathematics, University of Luxembourg and Department of Statistics, University of Oxford, tara.trauthwein@stats.ox.ac.uk. The author was supported by the Luxembourg National Research Fund (PRIDE17/1224660/GPS) and by the UK Engineering and Physical Sciences Research Council (EPSRC) grant (EP/T018445/1).}}
\date{}

\begin{document}
	\maketitle

\begin{abstract}
	In this work, we discuss new bounds for the normal approximation of multivariate Poisson functionals under minimal moment assumptions. Such bounds require one to estimate moments of so-called add-one costs of the functional. Previous works required the estimation of $4^{\text{th}}$ moments, while our result only requires $(2+\epsilon)$-moments, based on recent improvements introduced by (Trauthwein 2022). As applications, we show quantitative CLTs for two multivariate functionals of the Gilbert, or random geometric, graph. These examples were out of range for previous methods.
\end{abstract}

\noindent\textbf{Keywords:} Central Limit Theorem; Gilbert Graph; Malliavin Calculus; Multivariate Central Limit Theorem; Poincaré Inequality; Poisson Process; Stein's Method; Stochastic Geometry.\\

\noindent\textbf{Mathematics Subject Classification (2020)}: 60F05, 60H07, 60G55, 60D05

\section{Introduction}
The present paper establishes new distance bounds for multivariate Poisson functionals, allowing to derive quantitative Central Limit Theorems under minimal moment assumptions. It thus provides a multivariate counterpart to the improved univariate second-order Poincaré inequalities recently introduced in \cite{TT23}. As such, the paper extends some of the results from \cite{PZ10} and \cite{SY19}, providing comparable probabilistic inequalities but substantially reducing the moment conditions. The method used to achieve these minimal assumptions is to combine, as was done in \cite{PZ10}, Stein's and interpolation methods with Malliavin Calculus, and to make use of moment inequalities recently proposed in \cite{TT23}. Applications include, but are not limited to, the study of random geometric objects such as spatial random graphs.

The bounds in \Cref{thmMulti} are given in terms of the so-called \textbf{add-one cost} operator. Given a Poisson measure $\eta$ of intensity measure $\lambda$ on the $\sigma$-finite space $(\X,\lambda)$, let $F$ be a measurable, real-valued function of $\eta$. For any $x\in \X$, we define the add-one cost operator evaluated at $x$ by
\begin{equation}\label{eq:DefMalD}
\MalD_xF = \MalD_xF(\eta) := F(\eta + \delta_x) - F(\eta),
\end{equation}
where $\delta_x$ is the Dirac measure at $x$. This operator describes the change in the functional $F$ when a point $x$ is added to the measure $\eta$. Under additional assumptions on $F$, the operator $\MalD$ corresponds to the Malliavin derivative of $F$ at $x$.  The definition can be iterated to give the second derivative
\[
\DD_{x,y}F := \MalD_y(\MalD_xF).
\]
Our main result \Cref{thmMulti} allows one to derive quantitative CLTs for vector-valued $F=(F_1,\ldots,F_m)$ by controlling the covariance matrix $\cov(F)$ and $(2+\epsilon)$-moments of the terms $\mal_xF_i$ and $\DD_{x,y}F_i$ for $i=1,\ldots,m$. Previous multivariate bounds as in \cite[Thm.~1.1]{SY19} typically asked one to uniformly bound moments of the order $4+\epsilon$.

This type of bound relying on the add-one cost operators is particularly useful for quantities exhibiting a type of `local dependence', more generally known as \textbf{stabilization}. See e.g. \cite{LPS14,LRSY19,SY21} for further details on this topic. On a heuristic level, the first add-one cost $\mal_x F$ quantifies the amount of local change induced when adding the point $x$, while the second add-one cost $\DD_{x,y}F$ controls the dependence of points $x$ and $y$ which are further apart.

\Cref{thmMulti} provides bounds for distances of the type
\[
d(F,N) = \sup_{h \in \mathcal{H}} |\E h(F) - \E h(N)|
\]
where $F$ is a multivariate Poisson functional and $N$ a multivariate Gaussian with covariance matrix $C$. The distances we treat are the $d_2$ and the $d_3$ distances, where the test functions $h$ are chosen to be $\mathcal{C}^2$, resp. $\mathcal{C}^3$, with boundedness conditions (see \eqref{defd2} and \eqref{defd3} for precise definitions). The bound of the $d_2$ distance uses the multivariate \textbf{Stein method}, which comes at the detriment of needing the matrix $C$ to be positive-definite. The bound for the $d_3$ distance uses an \textbf{interpolation method} to circumvent this problem, but it comes at the cost of needing a higher degree of regularity in the test functions. In \cite{SY19}, the authors also provide a bound in the convex distance, where test functions are indicator functions of convex sets. Such a bound under minimal moment assumptions is out of reach for now, as it utilizes an involved recursive estimate which introduces new terms needing to be bounded by moments of the add-one costs. These terms cannot be treated with the currently known methods of reducing moment conditions.

The passage from the bounds achieved by Stein's method (resp. the interpolation method) to a bound involving add-one cost operators is achieved using \textbf{Malliavin Calculus}. The first combination of the two methods dates back to \cite{PN09,PNR2009} in a Gaussian context, and to \cite{PeccTaqq} in a Poisson context, and has since seen countless applications. The first appearance of a bound relying solely on moments of add-one costs was in the seminal paper \cite{LPS14}.

We study two applications in this paper in \Cref{thmGilMulti1,thmGilMulti2}, both relating to the \textbf{random geometric graph}, or \textbf{Gilbert graph}, the first study of which dates back to \cite{Gil61}. In this model, two points are connected if their distance is less than some threshold parameter. We study functionals of the type
\[
F = \sum_{\text{edges } e \text{ in } G_W} |e|^\alpha,
\]
where $G_W$ is the Gilbert graph restricted to vertices lying in some convex body $W$ and $|e|$ denotes the length of the edge $e$. We look at two types of vectors, one where the exponent $\alpha$ varies over different components, with fixed set $W$, and the other where the set $W$ varies for different components and $\alpha$ is fixed. In both settings we derive quantitative CLTs when we let the intensity of the underlying Poisson process grow to infinity. The setting with varying exponents has also been studied in \cite{TSRGilbert}, where a qualitative CLT has been derived in \cite[Thm.~5.2]{TSRGilbert} and the limit of the covariance matrix was given in \cite[Thm.~3.3]{TSRGilbert}. We provide a quantiative analogue to their result. In the setting with varying underlying sets, even the qualitative result is new.\\

\textit{Plan of the paper.} We present our main results in Section~\ref{secMain} and its applications in Section~\ref{secGil}. The proof of \Cref{thmMulti} is given in Section~\ref{secPMulti} and the ones for \Cref{thmGilMulti1,thmGilMulti2} in Section~\ref{secPGilMulti}. We introduce the necessary Malliavin operators and results about Malliavin Calculus and Poisson measures in Appendix~\ref{secApp}.

\textit{Acknowledgment.} I would like to thank Giovanni Peccati and Gesine Reinert for their support and helpful comments on this project.

\section{Main Result}\label{secMain}
We now present our main results. For a function $\phi:\R^m \rightarrow \R$ which is $k$ times continuously differentiable, denote by $\kder{\phi}{k}$ (resp. $\norm{\phi'}_\infty$ if $k=1$) the supremum of the absolute values of all $k^{\text{th}}$ partial derivatives, i.e.
\[
\kder{\phi}{k} := \max_{1 \leq i_1,\ldots,i_k \leq m} \sup_{x \in \R^m} \left|\frac{\partial^k \phi}{\partial x_{i_1} \ldots \partial x_{i_k}}(x)\right|.
\]
For a vector $x \in \R^m$, we denote by $\norm{x}$ the Euclidean norm of $x$.

 We can now introduce the distances to be used in this context. Define the following two sets of functions:
\begin{itemize}
	\item Let $\mathcal{H}^{(2)}_m$ be the set of all $\mathcal{C}^2$ functions $h:\R^m \rightarrow \R$ such that $|h(x)-h(y)| \leq \norm{x-y}$ for all $x,y\in \R^m$ and $\sup_{x\in \R^m} \norm{\opHess h(x)}_{op} \leq 1$;
	\item let $\mathcal{H}^{(3)}_m$ be the set of all $\mathcal{C}^3$ functions $h:\R^m\rightarrow\R$ such that $\kder{h}{2}$ and $\kder{h}{3}$ are bounded by $1$.
\end{itemize}
Here $\opHess h(x)$ denotes the Hessian matrix of $h$ at the point $x$ and $\norm{.}_\infty$ denotes the operator norm.

Let $X,Y$ be two $m$-dimensional random vectors. Define the $d_2$ and $d_3$ distances between $X$ and $Y$ as follows:
\begin{align}
	&d_2(X,Y) := \sup_{h \in \mathcal{H}^{(2)}_m} \left|\E h(X) - \E h(Y)\right|, \quad \text{if } \E \norm{X},\E\norm{Y} < \infty \label{defd2}\\
	&d_3(X,Y) := \sup_{h \in \mathcal{H}^{(3)}_m} \left|\E h(X) - \E h(Y)\right|, \quad \text{if } \E \norm{X}^2,\E\norm{Y}^2 < \infty. \label{defd3}
\end{align}

Recall the definition of the add-one cost from \eqref{eq:DefMalD}. We say that a measurable functional $F$ of $\eta$ is in $L^2(\p_\eta)$ if $\E F^2 <\infty$ and that it is in $\dom \mal$ if in addition
\[
\int_\X \E (\mal_x F)^2 \lambda(dx) < \infty.
\]

With these definitions, we can state the main result of this paper.

Formally define the terms $\zeta_1^{(p)},...,\zeta_4^{(q)}$ by

\begin{align*}
	&\zeta_1^{(p)} := \sum_{i,j=1}^m \left|C_{ij} - \cov(F_i,F_j)\right| \\
	&\zeta_2^{(p)} := 2^{2/p-1} \sum_{i,j=1}^m \left(\int_\X \left(\int_\X \E\left[|\DD_{x,y} F_i|^{2p}\right]^{1/2p}\E\left[|\DD_{x,y} F_j|^{2p}\right]^{1/2p}\lambda(dx)\right)^p\lambda(dy)\right)^{1/p} \\
	&\zeta_3^{(p)} := 2^{2/p} \sum_{i,j=1}^m \left(\int_\X \left(\int_\X \E\left[|\mal_{x} F_i|^{2p}\right]^{1/2p}\E\left[|\DD_{x,y} F_j|^{2p}\right]^{1/2p}\lambda(dx)\right)^p\lambda(dy)\right)^{1/p} \\
	&\zeta_4^{(q)} := m^{q-1} \sum_{i,j=1}^m \int_\X \E \left[|\mal_x F_i|^{q+1}\right]^{1/(q+1)} \E \left[|\mal_x F_j|^{q+1}\right]^{1-1/(q+1)} \lambda(dx).
\end{align*}

Then the following statement holds.

\begin{thm}\label{thmMulti}
	Let $(\X,\lambda)$ be a $\sigma$-finite measure space and $\eta$ be a $(\X,\lambda)$-Poisson measure. Let $m \geq 1$ and let $F=(F_1,...,F_m)$ be an $\R^m$-valued random vector such that for $1 \leq i \leq m$, we have $F_i \in L^2(\p_\eta) \cap \dom \mal$ and $\E F_i=0$. Let $C=(C_{ij})_{1 \leq i,j \leq m}$ be a symmetric positive-semidefinite matrix and let $X \sim \mathcal{N}(0,C)$. Then for all $p,q \in [1,2]$,
	\[
	d_3(F,X) \leq \tfrac{1}{2}(\zeta_1^{(p)} + \zeta_2^{(p)} + \zeta_3^{(p)}) + \zeta_4^{(q)}.
	\]
	If moreover the matrix $C$ is positive-definite, then for all $p,q \in [1,2]$,
	\[
	d_2(F,X) \leq \norm{C^{-1}}_{op} \norm{C}_{op}^{1/2} (\zeta_1^{(p)}+\zeta_2^{(p)}+\zeta_3^{(p)}) + \left(2\norm{C^{-1}}_{op} \norm{C}_{op}^{1/2} \vee \tfrac{\sqrt{2\pi}}{8} \norm{C^{-1}}_{op}^{3/2} \norm{C}_{op}\right) \zeta_4^{(q)}.
	\]
\end{thm}

The proof of these bounds is located in \Cref{secPMulti}. It relies on the work \cite{PZ10}, where the bound on the $d_2$ distance uses the Malliavin-Stein method, whereas the one for the $d_3$ distance uses an interpolation technique taken from the context of spin glasses (see \cite{Tal03}). The advantage of using an interpolation technique is that the limiting covariance matrix $C$ does not necessarily need to be positive-definite, but it comes at the cost of a higher regularity requirement for the test functions $h$. Building on the results of \cite{PZ10}, we adapt the univariate improvement of \cite{BOPT} to a multivariate setting and make use of the $p$-Poincaré inequality introduced in \cite[(4.7) in Remark~4.4]{TT23} (see also \Cref{prop:ppoin} in the Appendix).

\section{Applications}\label{secGil}

In this section, we study two multivariate functionals of the Gilbert graph. Both of these functionals consist of sums of power-weighted edge-lengths --- in the first functional, we vary the powers of the edges-lengths, and in the second functional, we restrict the graph to different domains.

We start by setting the general framework. Let $t\geq 1$ and $(\epsilon_t)_{t>0}$ a sequence of positive real numbers such that $\epsilon_t \rightarrow 0$ as $t \rightarrow \infty$. Define $\eta^t$ to be a Poisson measure on $\R^d$ of intensity $t\, dx$. For any convex body (i.e. a convex compact set with non-empty interior) $W \subset \R^d$ and any exponent $\alpha \in \R$, define
\begin{equation}\label{eq:DefLat}
\lat(W) := \frac{1}{2}\sum_{x,y \in \eta^t_{|W}} \ind{0 < \norm{x-y} < \epsilon_t} \norm{x-y}^{\alpha},
\end{equation}
where $\eta^t_{|W}$ denotes the restriction of $\eta$ to the set $W$. The quantity $\lat(W)$ is the sum of all edge-lengths to the power $\alpha$ in the Gilbert graph with parameter $\epsilon_t>0$ and points in $W$.

\begin{figure}[h]
	\begin{subfigure}[t]{.45\textwidth}
		\includegraphics[width=\textwidth]{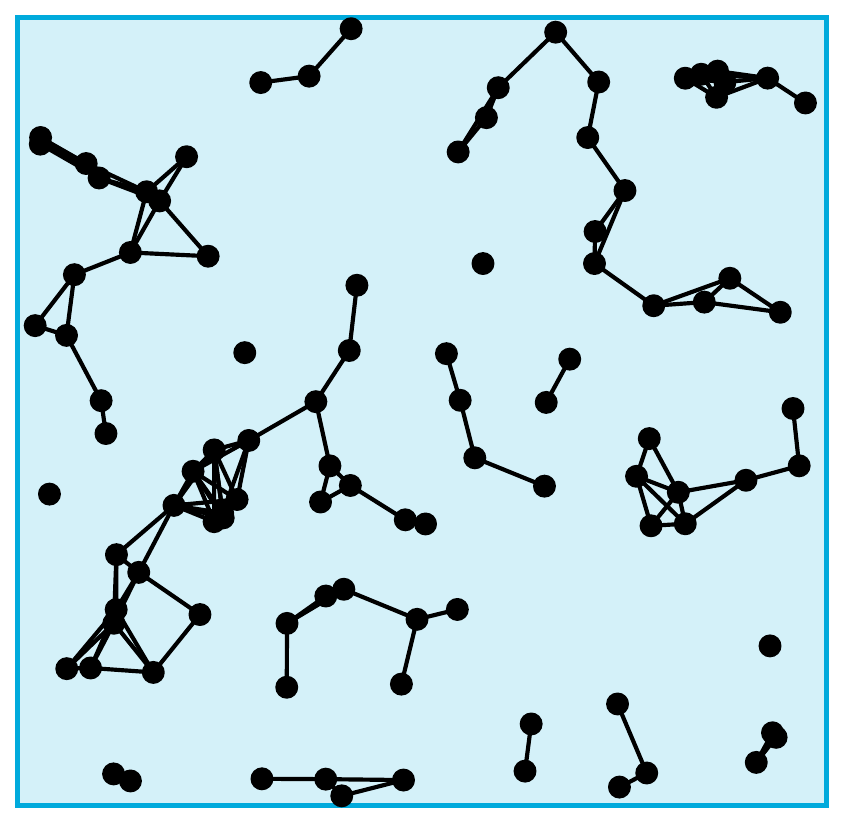}
		\caption{In the first application, we fix one Gilbert graph and look at the vector whose components are the sums of power-weighted edge-lengths, for different powers.}
	\end{subfigure}%
	\hspace{.05\textwidth}%
	\begin{subfigure}[t]{.45\textwidth}
		\includegraphics[width=\textwidth]{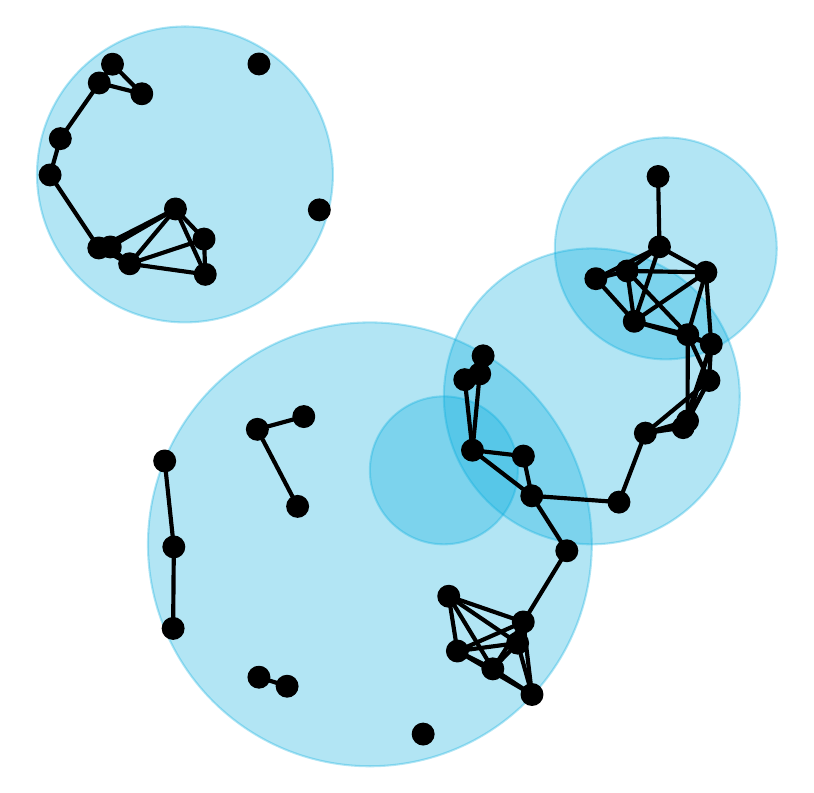}
		\caption{In the second application, we construct the Gilbert graph in multiple domains and study the vector whose components are the sums of power-weighted edge-lengths in the different domains.}
	\end{subfigure}
\end{figure}

\subsection{Varying the Exponents}

 We fix $W \subset \R^d$ to be a convex body and consider real numbers $\alpha_1,...,\alpha_m$ such that $\alpha_i+\alpha_j>-d$ for all $i,j \in \{1,...,m\}$. For every $1 \leq i \leq m$, define $\lati{i}(W)=\lati{i}$ as in \eqref{eq:DefLat} and set
\[
\tillati{i} := \left(t\epsilon_t^{\alpha_i + d/2} \vee t^{3/2}\epsilon_t^{\alpha_i+d}\right)^{-1} \left(\lati{i}-\E \lati{i}\right).
\]

Set furthermore
\[
\sigma_{ij}^{(1)} := \frac{d\kappa_d}{2|\alpha_i+\alpha_j+d|} \quad \text{and} \quad \sigma_{ij}^{(2)} := \frac{d^2\kappa_d^2}{(\alpha_i+d)(\alpha_j+d)},
\]
where $\kappa_d$ denotes the volume of a unit ball in $\R^d$, and define the matrix $C=(C_{ij})_{1 \leq i,j \leq m}$ by $C_{ij} = |W|c_{ij}$, where
\[
c_{ij} := 
\begin{cases*}
	\sigma_{ij}^{(1)} & if $\displaystyle\lim_{t\rightarrow\infty} t\epsilon_t^d = 0$ \\
	\left(\sigma_{ij}^{(1)} + \theta \sigma_{ij}^{(2)} \right) & if $\displaystyle\lim_{t\rightarrow\infty} t\epsilon_t^d = \theta \leq 1$ \\
	\left(\tfrac{1}{\theta}\sigma_{ij}^{(1)} + \sigma_{ij}^{(2)} \right) & if $\displaystyle\lim_{t\rightarrow\infty} t\epsilon_t^d = \theta > 1$ \\
	\sigma_{ij}^{(2)} & if $\displaystyle\lim_{t\rightarrow\infty} t\epsilon_t^d = \infty$, \\
\end{cases*}
\]
and $|W|$ denotes the volume of $W$. Defining the vector $\tilde{L}_t$ as 
\[
\tilde{L}_t := \left(\tillati{1},...,\tillati{m}\right),
\]
we recall that a CLT for $\tilde{L}_t$ has been shown in \cite[Thm.~5.2]{TSRGilbert}, as well as convergence of the covariance matrix of $\tilde{L}_t$ to $C$ in \cite[Thm.~3.3]{TSRGilbert}. The matrix $C$ is positive-definite in the sparse and thermodynamic regime (i.e. if $t\epsilon_t^d \rightarrow 0$ or $t\epsilon_t^d \rightarrow c>0$), while it is singular in the dense regime (i.e. if $t\epsilon_t^d \rightarrow \infty$), see \cite[Prop.~3.4]{TSRGilbert}.
Define also
\begin{equation}\label{eq:betadef}
\beta_{ij}^{(t)} := \frac{\sigma_{ij}^{(1)} + \sigma_{ij}^{(2)}t\epsilon_t^d}{1 \vee t\epsilon_t^d}
\end{equation}
and note that $\beta_{ij}^{(t)} \rightarrow c_{ij}$ as $t \rightarrow \infty$.
\begin{thm}\label{thmGilMulti1}
	Assume that $t^2\epsilon_t^d \rightarrow \infty$ as $t \rightarrow \infty$. Let $X\sim\mathcal{N}(0,C)$ be a centred Gaussian with covariance matrix $C$. Then
	\begin{itemize}
		\item if $\alpha_1,...,\alpha_m>-\frac{d}{4}$, there is a constant $c_1>0$ such that for all $t\geq 1$ large enough
		\begin{equation}\label{eqMultiBound1}
			d_3(\tilde{L}_t, X) \leq c_1 \left(\epsilon_t + \max_{1 \leq i,j \leq m}\left|\beta_{ij}^{(t)}-c_{ij}\right| + \left(t^{-1/2} \vee (t^2\epsilon_t^d)^{-1/2}\right)\right).
		\end{equation}
		\item if $-\frac{d}{2}<\min\{\alpha_1,...,\alpha_m\} \leq -\frac{d}{4}$, then for any $1<p<-\frac{d}{2} \min\{\alpha_1,...,\alpha_m\}^{-1}$, there is a constant $c_2>0$ such that for all $t\geq 1$ large enough
		\begin{equation}\label{eqMultiBound2}
			d_3(\tilde{L}_t, X) \leq c_2 \left(\epsilon_t + \max_{1 \leq i,j \leq m} \left|\beta_{ij}^{(t)}-c_{ij}\right| + \left(t^{-1+1/p} \vee (t^2\epsilon_t^d)^{-1+1/p}\right)\right).
		\end{equation}
		If $\lim_{t \rightarrow \infty} t\epsilon_t^d < \infty$, then the bounds \eqref{eqMultiBound1} and \eqref{eqMultiBound2} apply to $d_2(\tilde{L}_t,X)$ as  well for different constants $c_1,c_2>0$.
	\end{itemize}
\end{thm}
The proof of this theorem can be found in \Cref{secPGilMulti}. The bounds given in \Cref{thmGilMulti1} vary according to the limit of $t\epsilon^d$. We give a precise discussion of the bounds \eqref{eqMultiBound2} in Table~\ref{table:Multi1}. The bounds for \eqref{eqMultiBound1} follow when setting $p=2$.

{\renewcommand{\arraystretch}{2.2}
	\begin{table}[h!]
		\noindent\begin{tabular}{|l|c|l|c|}
			\hline
			& $\displaystyle c_{ij}$ & speed of convergence in \eqref{eqMultiBound2} & bound holds for \\
			\hline
			$\displaystyle t\epsilon^d_t \rightarrow 0$ & $\displaystyle\sigma_{ij}^{(1)}$ & $\displaystyle\epsilon_t + \sigma^{(2)}_* t\epsilon_t^d + \left(t^2\epsilon_t^d\right)^{-1+1/p}$ & $d_2$ and $d_3$ distances \\
			\hline
			$\displaystyle t\epsilon^d_t \rightarrow \theta \leq 1$& $\displaystyle\sigma_{ij}^{(1)} + \theta \sigma_{ij}^{(2)}$ & $\displaystyle\epsilon_t + \sigma^{(2)}_* \left|\theta - t\epsilon_t^d\right| + t^{-1+1/p}$ & $d_2$ and $d_3$ distances \\
			\hline
			$\displaystyle t\epsilon^d_t \rightarrow \theta>1$ & $\displaystyle\tfrac{1}{\theta}\sigma_{ij}^{(1)} + \sigma_{ij}^{(2)}$ & $\displaystyle\epsilon_t + \sigma^{(1)}_* \left|\tfrac 1 \theta - \tfrac{1}{t\epsilon_t^d}\right| + t^{-1+1/p}$ & $d_2$ and $d_3$ distances \\
			\hline
			$\displaystyle t\epsilon^d_t \rightarrow \infty$ & $\displaystyle \sigma_{ij}^{(2)}$ & $\displaystyle \epsilon_t + \sigma^{(1)}_* \tfrac{1}{t\epsilon_t^d} + t^{-1+1/p}$ & $d_3$ distance \\
			\hline
		\end{tabular}
		\caption{The different bounds depending on the limit of $t\epsilon_t^d$. Here we define $\sigma^{(k)}_*:= \max_{1 \leq i,j \leq m} \sigma_{ij}^{(k)}$.\label{table:Multi1}}
\end{table}}

\subsection{Varying the Domains}
 Fix $\alpha>-\frac{d}{2}$ to be a real number and let $W_1,...,W_m \subset \R^d$ be convex bodies. For every $1 \leq i \leq m$, define $F^{(i)}_t := \lat(W_i)$ using the definition \eqref{eq:DefLat} and set
\[
\tilde{F}_t^{(i)} := t^{-1} \epsilon_t^{-\alpha-d/2} \left(\tfrac{1}{2} \sigma_1  + \sigma_2 t\epsilon_t^d\right)^{-1/2} \left(F^{(i)}_t - \E F^{(i)}_t\right)
\]
and
\[
\tilde{F}_t := \left(\tilde{F}_t^{(1)},...,\tilde{F}_t^{(m)}\right),
\]
where
\[
\sigma_1 := \frac{d\kappa_d}{d+2\alpha} \quad \text{and} \quad \sigma_2 := \left(\frac{d\kappa_d}{\alpha+d}\right)^2.
\]
Define the matrix $C=(C_{ij})_{1 \leq i,j\leq m}$ by 
\[
C_{ij} := |W_i \cap W_j|.
\]
\begin{thm}\label{thmGilMulti2}
	Under the above conditions, the matrix $C$ is the asymptotic covariance matrix of the $m$-dimensional random vector $\tilde{F}_t$.
	
	Moreover, assume that $t^2\epsilon_t^d \rightarrow \infty$ as $t \rightarrow \infty$. Let $X\sim\mathcal{N}(0,C)$ be a centred Gaussian with covariance matrix $C$. Then 
	\begin{itemize}
		\item if $\alpha>-\frac{d}{4}$, there is a constant $c_1>0$ such that for all $t \geq 1$ large enough,
		\begin{equation}\label{eqMultiBound3}
			d_3(\tilde{F}_t, X) \leq c_1 \left(\epsilon_t + \left(t^{-1/2} \vee (t^2\epsilon_t^d)^{-1/2}\right)\right).
		\end{equation}
		\item if $-\frac{d}{2}<\alpha \leq -\frac{d}{4}$, then for any $1<p<-\frac{d}{2\alpha}$, there is a constant $c_2>0$ such that for all $t \geq 1$ large enough,
		\begin{equation}\label{eqMultiBound4}
			d_3(\tilde{F}_t, X) \leq c_2 \left(\epsilon_t + \left(t^{-1+1/p} \vee (t^2\epsilon_t^d)^{-1+1/p}\right)\right).
		\end{equation}
		If the matrix $C$ is positive definite, then the bounds \eqref{eqMultiBound3} and \eqref{eqMultiBound4} apply to $d_2(\tilde{F}_t,X)$ as  well for different constants $c_1,c_2>0$.
	\end{itemize}
\end{thm}
The speed of convergence varies according to the asymptotic behaviour of $t\epsilon_t^d$. In particular, one has
\begin{equation}\label{eqGilCases}
	t^{-1+1/p} \vee (t^2\epsilon_t^d)^{-1+1/p} =
	\begin{cases*}
		(t^2\epsilon_t^d)^{-1+1/p} & if $t\epsilon_t^d \rightarrow 0$ (sparse regime) \\
		t^{-1+1/p} & if $t\epsilon_t^d \rightarrow \theta>0$ (thermodynamic regime) \\
		t^{-1+1/p} & if $t\epsilon_t^d \rightarrow \infty$ (dense regime). \\
	\end{cases*}
\end{equation}
The proof of \Cref{thmGilMulti2} can be found in \Cref{secPGilMulti}.

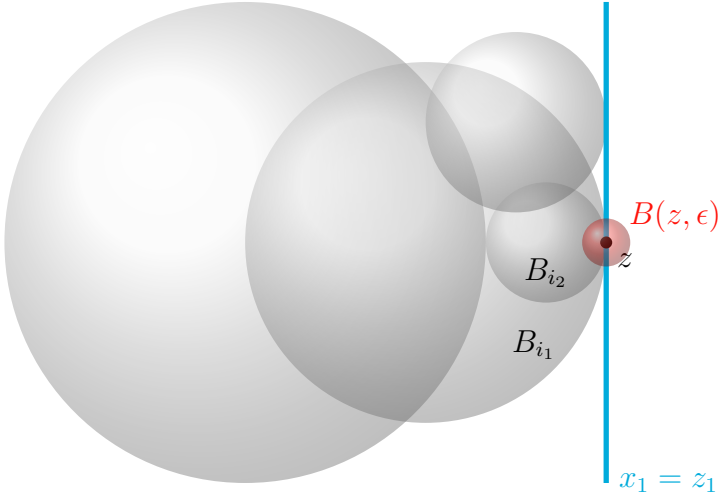
\begin{figure}[h]
	\begin{tikzpicture}[scale=.8]
		\shade[ball color = lGrey!40, opacity = 0.4] (0,0) circle (4);
		\shade[ball color = lGrey!40, opacity = 0.4] (3,0) circle (3);
		\shade[ball color = lGrey!40, opacity = 0.4] (5,0) circle (1);
		\shade[ball color = lGrey!40, opacity = 0.4] (4.5,2) circle (1.5);
		\node at (4.8,-1.7) {$B_{i_1}$};
		\node at (5,-.5) {$B_{i_2}$};
		\draw[line width=2pt, color=uniBlue] (6,-4) node[right] {$x_1=z_1$} -- ++(0,8);
		\fill[black] (6,0) circle (.1) node[anchor=north west] {$z$};
		\shade[ball color = uniRed, opacity = 0.4] (6,0) circle (.4);
		\node[anchor=south west, color=uniRed] at (6.2,0) {$B(z,\epsilon)$};
	\end{tikzpicture}
	\caption{Constructing a sequence from a collection of distinct balls}
	\label{fig:balls}
\end{figure}

\begin{rem}\label{remMulti2}
	The question whether $C$ is positive-definite is not entirely straightforward, but some things can be said. For a vector $x\in \R^m$, we have
	\[
	x^TCx = \sum_{i,j=1}^{m} x_ix_j \int_{\R^d} \1_{W_i}(z) \1_{W_j}(z) dz = \int_{\R^d} \left(\sum_{i=1}^{m} x_i \1_{W_i}(z) \right)^2 dz.
	\]
	Since $W_1,\ldots,W_m$ form a collection of convex bodies, it is clearly a necessary and sufficient condition that the family of indicators $\1_{W_1}(z),...,\1_{W_m}(z)$ is linearly independent in $L^2(\R^d)$, which translates (to some extent) to none of the sets being obtainable from the other sets via certain combinations of unions, intersections and complements, disregarding sets of measure zero.
	
	%	To be sufficient, one would need that the implication
	%	\begin{equation}\label{eqaeimpliese}
		%	\sum_{i=1}^{m} x_i \1_{W_i}(z)=0 \quad \text{for a.e. } z\in \R^d \Rightarrow	\sum_{i=1}^{m} x_i \1_{W_i}(z)=0 \quad \text{for every } z\in\R^d.	
		%	\end{equation}
	%	holds for all $x \in \R^m$ and all convex bodies $W_1,...,W_m \subset \R^d$. This question is work in progress.
	
	A simple sufficient condition for positive definiteness is that for each $1 \leq i \leq m-1$, we have
	\begin{equation}\label{eqConvCond}
		W_i \setminus \bigcup_{j>i} W_j \neq \emptyset,
	\end{equation}
	i.e. the sets $W_1,...,W_m$ can be ordered in a sequence such that each set has a point not included in subsequent sets. This implies that the family of indicators $\1_{W_1}(z),...,\1_{W_m}(z)$ is linearly independent, and it also entails positive-definiteness of $C$. Indeed, let $z_1 \in W_1 \setminus \bigcup_{j>1} W_j$. Since $\bigcup_{j>1} W_j$ is closed, there is an open set $U \subset\left(\bigcup_{j>1} W_j\right)^c$ such that $z_1 \in U$. As $W_1$ is a convex body, the intersection $V$ of the interior of $W_1$ with $U$ is open and non-empty. Assume that $x^TCx=0$, then the function $f=x_1\1_{W_1} + \ldots + x_m\1_{W_m}$ is zero almost everywhere. In particular, it is constant $f \equiv x_1$ on $V \subset W_1 \setminus \bigcup_{j>1} W_j$, hence we must have $x_1=0$. One can now iterate this argument to show that $x_2=\ldots=x_m=0$.
	
	\begin{figure}[h]
		\begin{tikzpicture}[scale=2]
			\draw[fill=uniRed] (0,0) rectangle +(.8,1) node[pos=.5] {$W_1$};
			\draw[dotted] (.8,1) -- ++(.2,0) -- ++(0,-1) -- ++(-.2,0);
			\draw[fill=uniBlue] (1.5,1) rectangle +(1,-.8) node[pos=.5] {$W_2$};
			\draw[dotted] (1.5,.2) -- ++(0,-.2) -- ++(1,0) -- ++(0,.2);
			\draw[fill=green] (3.2,0) rectangle +(.8,1) node[pos=.5] {$W_3$};
			\draw[dotted] (3.2,0) -- ++(-.2,0) -- ++(0,1) -- ++(.2,0);
			\draw[fill=yellow] (4.5,0) rectangle +(1,.8) node[pos=.5] {$W_4$};
			\draw[dotted] (4.5,.8) -- ++(0,.2) -- ++(1,0) -- ++(0,-.2);
			\draw[-{Latex[length=3mm,width=5mm]},line width=2pt] (5.8,.5) -- (6.5,.5);
			\draw[fill=uniRed,opacity=.3] (6.8,0) rectangle +(.8,1);
			\draw[fill=uniBlue,opacity=.3] (6.8,1) rectangle +(1,-.8);
			\draw[fill=green,opacity=.3] (7,0) rectangle +(.8,1);
			\draw[fill=yellow,opacity=.3] (6.8,0) rectangle +(1,.8);
		\end{tikzpicture}
		\caption{Four convex bodies which are being superimposed on the right. Every point is covered by at least two sets, but no set can be obtained by combining unions and intersections of the other sets.}
		\label{fig:fourBodies}
	\end{figure}
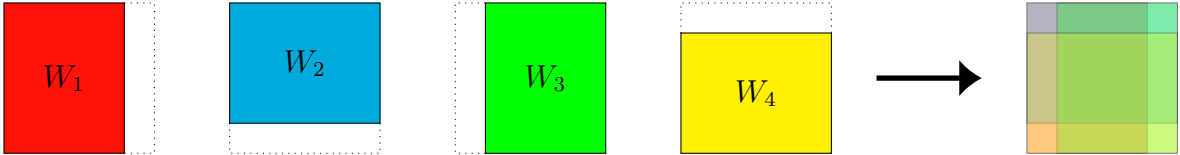
	
	If the collection $W_1,\ldots,W_m$ consists of distinct balls, we can prove positive-definiteness of $C$ using condition \eqref{eqConvCond}. Indeed, let $B_1,\ldots,B_m$ be a collection of closed distinct balls in $\R^d$. The union $\bigcup_{i\geq 1} B_i$ is closed and bounded, and admits thus at least one point $z$ whose first coordinate achieves the maximum over all first coordinates of points in $\bigcup_{i\geq 1} B_i$. Let $B_{i_1},\ldots,B_{i_n}$ be the balls tangent to the hyperplane $x_1=z_1$, sorted by decreasing radius (i.e. $B_{i_1}$ is the largest ball among $B_{i_1},\ldots,B_{i_n}$). There is an open ball $B(z,\epsilon) \subset \left(\bigcup_{j \neq i_1,\ldots,i_n}B_j\right)^c$, and the intersection $B(z,\epsilon) \cap B_{i_1}$ contains a point not included in $B_{i_2},\ldots,B_{i_n}$. The ball $B_{i_1}$ can thus be made the first ball in the sequence. This construction can be iterated to satisfy condition \eqref{eqConvCond}. See Figure~\ref{fig:balls} for an illustration.
	
	The condition \eqref{eqConvCond} is however not necessary for the matrix $C$ to be positive definite. Indeed, consider a superposition of the sets $W_1,\ldots,W_4$ depicted in Figure~\ref{fig:fourBodies}. One can show easily that if $x_1\1_{W_1} + \ldots + x_4\1_{W_4} =0$ almost everywhere, then $x_1=\ldots=x_4=0$, hence the indicators are linearly independent in $L^2(\R^d)$ and the resulting matrix is positive-definite. However, these sets do not fulfill condition \eqref{eqConvCond}.
\end{rem}

\section{Proof of \Cref{thmMulti}}\label{secPMulti}

The proof of \Cref{thmMulti} uses both an interpolation technique and the multivariate Stein method. For a positive-definite symmetric matrix $C=(C_{ij})_{1 \leq i,j \leq m}$ and a function $g:\R^m \rightarrow \R$, the multivariate Stein equation is given by
\begin{equation}\label{eqMulStein}
	g(x) - \E g(X) = \langle x , \nabla f(x) \rangle_{\R^m} - \langle C , \opHess f(x) \rangle_{H.S.},
\end{equation}
for $x \in \R^m$ and $X \sim \mathcal{N}(0,C)$. The inner product $\langle .,. \rangle_{H.S.}$ is the Hilbert-Schmidt inner product defined as $\langle A,B \rangle_{H.S.}:= Tr(AB^T)$ for real $m \times m$ matrices $A$ and $B$ and where $Tr(.)$ denotes the trace function. If $g\in \mathcal{C}^2(\R^m)$ has bounded first- and second-order partial derivatives, then a solution to \eqref{eqMulStein} is given by
\begin{equation}\label{eqcansol}
	f_g(x) := \int_0^1 \frac{1}{2t} \E\left[g(\sqrt{t}x + \sqrt{1-t}X) - g(X)\right]\, dt.
\end{equation}
The solution $f_g$ satisfies the following bounds:
\begin{align}
	&\kder{f_h}{2} \leq \norm{C^{-1}}_{op} \norm{C}_{op}^{1/2} \norm{g'}_\infty \label{eqH1} \\
	\intertext{and}
	&\kder{f_g}{3} \leq \frac{\sqrt{2\pi}}{4} \norm{C^{-1}}_{op}^{3/2} \norm{C}_{op} \sup_{x \in \R^m} \norm{\opHess g(x)}_{op}, \label{eqH2}
\end{align}
where $\opHess g(x)$ denotes the Hessian matrix of the function $g$ evaluated at $x$ and $\norm{.}_{op}$ denotes the operator norm.
These results can be found in \cite[Lemma~2.17]{PZ10}.

Before we can give the proof of \Cref{thmMulti}, we also need some technical estimates which improve the corresponding estimates given in \cite{PZ10}. The first lemma is an extension of Lemma~3.1 in \cite{PZ10}: contrary to what was done in \cite{PZ10}, we bound the rest term $R_x$ in the development below by a $q^{\text{th}}$ power of the add-one costs, with $q \in [1,2]$. Lemma~3.1 in \cite{PZ10} corresponds to the choice of $q=2$.
\begin{lem}\label{lem3.1}
	Let $F=(F_1,...,F_m)$ for some $m \geq 1$, where $F_i \in L^2(\p_\eta) \cap \dom \mal$ and $\E F_i=0$ for $1 \leq i \leq m$. Then for all $\phi \in \mathcal{C}^2(\R^m)$ with $\norm{\phi'}_\infty,\kder{\phi}{2}<\infty$, it holds that for a.e. $x\in\X$ and all $q\in[1,2]$,
	\[
	\mal_x \phi(F) = \sum_{i=1}^{m} \frac{\partial \phi}{\partial x_i}(F) \mal_x F_i + R_x,
	\]
	where
	\[
	|R_x| \leq \left(2\norm{\phi'}_\infty \vee \tfrac{1}{2}\kder{\phi}{2}\right) \left(\sum_{i=1}^{m} |\mal_x F_i|\right)^q.
	\]
\end{lem}
\begin{proof}
	One has that
	\begin{align}
		\mal_x \phi(F) &= \phi(F + \mal_xF) - \phi(F) \notag\\
		&= \int_0^1 \sum_{i=1}^{m} \frac{\partial \phi}{\partial x_i} (F+t\mal_x F_i) \cdot \mal_x F_i\, dt \notag\\
		&= \sum_{i=1}^{m} \frac{\partial \phi}{\partial x_i} (F) \cdot \mal_x F_i + R_x,
	\end{align}
	where
	\[
	R_x = \sum_{i=1}^{m} \int_0^1 \left(\frac{\partial \phi}{\partial x_i}(F+t\mal_xF) - \frac{\partial \phi}{\partial x_i}(F)\right) \cdot \mal_x F_i \,dt.
	\]
	Note that
	\[
	\left|\frac{\partial \phi}{\partial x_i}(F+t\mal_xF) - \frac{\partial \phi}{\partial x_i}(F)\right| \leq \left(2 \norm{\phi'}_\infty\right) \wedge \left(t\kder{\phi}{2} \norm{\mal_xF}\right)
	\]
	by the mean value theorem. Hence
	\begin{align}
		|R_x| &\leq \left(2 \norm{\phi'}_\infty \sum_{i=1}^{m} |\mal_x F_i|\right) \wedge \left(\tfrac{1}{2} \kder{\phi}{2} \norm{\mal_xF} \sum_{i=1}^{m} |\mal_xF_i|\right) \notag\\
		&\leq \left(2\norm{\phi'}_\infty \vee \frac{1}{2}\kder{\phi}{2}\right) \left[\sum_{i=1}^{m} |\mal_xF_i| \wedge \left(\sum_{i=1}^{m} |\mal_xF_i|\right)^2\right]\notag\\
		&\leq \left(2\norm{\phi'}_\infty \vee \frac{1}{2}\kder{\phi}{2}\right) \left(\sum_{i=1}^{m} |\mal_xF_i|\right)^q.
	\end{align}
\end{proof}
The next lemma is an improvement of Lemma~4.1 in \cite{PZ10}, where the improvement comes from the fact that we use \Cref{lem3.1} in the final step. For the definition of the operator $\opL$, see \eqref{eq:DefopL}.
\begin{lem}\label{lem4.1}
	Let $m \geq 1$ and for $0 \leq i \leq m$, let $F_i \in L^2(\p_\eta) \cap \dom\mal$ and assume $\E F_i=0$. Then for all $g \in \mathcal{C}^2(\R^m)$ such that $\norm{g'}_\infty,\kder{g}{2}<\infty$, we have
	\[
	\E g(F_1,...,F_m)F_0 = \E \sum_{i=1}^{m} \frac{\partial g}{\partial x_i}(F_1,...,F_m) \langle \mal F_i, -\mal \opL F_0 \rangle_{L^2(\lambda)} + R,
	\]
	where for all $q\in[1,2]$,
	\[
	 |R| \leq \left(2\norm{g'}_\infty \vee \tfrac{1}{2} \kder{g}{2}\right) \int_\X \E \left(\sum_{k=1}^{m} |\mal_xF_k|\right)^q |\mal_x\opL F_0| \lambda(dx).
	\]
\end{lem}
\begin{proof}
	Using integration by parts (see \Cref{lem:opL}) as detailed in the proof of \cite[Lemma~4.1]{PZ10}, we derive that
	\begin{align}
		\E g(F_1,...,F_m)F_0 &= \E \langle \mal g(F_1,...,F_m),-\mal \opL F_0\rangle_{L^2(\lambda)} \\
		&= \E \sum_{i=1}^{m} \frac{\partial g}{\partial x_i} (F_1,...,F_m) \langle \mal F_i, -\mal \opL F_0 \rangle_{L^2(\lambda)} + \E \langle R_*,-\mal \opL F_0 \rangle_{L^2(\lambda)},
	\end{align}
	where $R_*$ is a rest term satisfying the bound in \Cref{lem3.1}. The claimed result follows with $R:= \E \langle R_*,-\mal \opL F_0 \rangle_{L^2(\lambda)}$ .
\end{proof}
The next proposition is an extension of Theorems~3.3 and 4.2 in \cite{PZ10} and exploits much of the arguments rehearsed in the proofs of these theorems, which are combined with the content of \Cref{lem3.1,lem4.1}.
\begin{prop}\label{propMulti}
	Let $m \geq 1$ and let $F=(F_1,...,F_m)$ be an $\R^m$-valued random vector such that for $1\leq i \leq m$, we have $F_i \in L^2(\p_\eta) \cap \dom \mal$ and $\E F_i=0$. Let $C=(C_{ij})_{1\leq i,j \leq m}$ be a symmetric positive-semidefinite matrix and let $X\sim \mathcal{N}(0,C)$. Then for all $q \in [1,2]$,
	\begin{multline}
		d_3(F,X) \leq \tfrac{1}{2} \sum_{i,j=1}^{m} \E \left[|C_{ij} - \langle \mal F_i, -\mal \opL F_j \rangle_{L^2(\lambda)}|\right] \\+ \sum_{i=1}^{m} \int_\X \E  \left[\left(\sum_{j=1}^{m} |\mal_x F_j|\right)^q |\mal_x \opL F_i|\right] \lambda(dx).
	\end{multline}
	If moreover the matrix $C$ is positive-definite, then for all $q \in [1,2]$,
	\begin{align}
		d_2(F,X) \leq& \norm{C^{-1}}_{op} \norm{C}_{op}^{1/2} \sum_{i,j=1}^{m} \E \left[|C_{ij} - \langle \mal F_i, -\mal \opL F_j \rangle_{L^2(\lambda)}|\right] \notag\\
		&+ \left(2\norm{C^{-1}}_{op} \norm{C}_{op}^{1/2} \vee \frac{\sqrt{2\pi}}{8}\norm{C^{-1}}_{op}^{3/2} \norm{C}_{op} \right) \notag\\
		&\hspace{5cm}\sum_{i=1}^{m} \int_\X \E  \left[\left(\sum_{j=1}^{m} |\mal_x F_j|\right)^q |\mal_x \opL F_i|\right] \lambda(dx).
	\end{align}
\end{prop}
\begin{proof}
	To show the bound on the $d_3$ distance, we proceed as in the proof of Theorem~4.2 in \cite{PZ10}, but we replace the use of Lemma~4.1 therein with our \Cref{lem4.1}. Indeed, we only need to show that
	\begin{multline}
		\left|\E[\phi(F)] - \E[\phi(X)]\right| \leq \tfrac{1}{2} \kder{\phi}{2} \sum_{i,j=1}^{m} \E \left[|C_{ij} - \langle \mal F_i, -\mal \opL F_j \rangle_{L^2(\lambda)}\right] \\+ \tfrac{1}{2} \left(2\kder{\phi}{2} \vee \tfrac{1}{2} \kder{\phi}{3}\right) \sum_{i=1}^{m} \int_\X \E  \left[\left(\sum_{j=1}^{m} |\mal_x F_j|\right)^p |\mal_x \opL F_i|\right] \lambda(dx)
	\end{multline}
	for any $\phi \in \mathcal{C}^3(\R^m)$ with bounded second and third partial derivatives. Defining
	\[
	\psi(t):= \E \phi(\sqrt{1-t}F+\sqrt{t}X),
	\]
	it is clear that 
	\[
	|\E \phi(F) - \E \phi(X)| \leq \sup_{t \in (0,1)} |\psi'(t)|.
	\]
	Defining moreover
	\[
	\phi_i^{t,b}(x) := \frac{\partial \phi}{\partial x_i}(\sqrt{1-t}x+\sqrt{t}b)
	\]
	for any vector $b \in \R^m$, it is shown in the proof of \cite[Thm.~4.2]{PZ10} that $\psi'(t)$ can be written as
	\[
	\psi'(t)=\frac{1}{2\sqrt{t}} \mathcal{A} - \frac{1}{2\sqrt{1-t}} \mathcal{B},
	\]
	where
	\[
	\mathcal{A} = \sqrt{t} \sum_{i,j=1}^{m} C_{ij} \E \frac{\partial^2\phi}{\partial x_i \partial x_j} (\sqrt{1-t}F + \sqrt{t}X)
	\]
	and
	\[
	\mathcal{B} = \sum_{i=1}^{m} \E \frac{\partial \phi}{\partial x_i} (\sqrt{1-t}F + \sqrt{t}X) F_i.
	\]
	Conditioning on $X$ in $\mathcal{B}$, one can apply \Cref{lem4.1} and deduce that
	\[
		\mathcal{B} = \sqrt{1-t} \sum_{i,j=1}^{m} \E \frac{\partial^2 \phi}{\partial x_i \partial x_j} (\sqrt{1-t}F + \sqrt{t}X) \langle \mal F_i, -\mal\opL F_j \rangle + \sum_{i=1}^{m} \E \left[R_X^i\right],
	\]
	where $R_X^i$ satisfies
	\[
		|R_X^i| \leq \left(\tfrac{1}{2} \norm{(\phi_i^{t,X})^{(2)}}_\infty \vee 2 \norm{(\phi_i^{t,X})'}_\infty\right) \int_\X \E \left(\sum_{j=1}^{m} |\mal_x F_j|\right)^q |\mal_x \opL F_i| \lambda(dx).
	\]
	It suffices now to observe that
	\begin{align}
		&\left|\frac{\partial^2 \phi}{\partial x_i \partial x_j} (\sqrt{1-t}F + \sqrt{t}X)\right| \leq \kder{\phi}{2} \\
		\intertext{and}
		&\left|\frac{\partial \phi_i^{t,b}}{\partial x_j} (x)\right| = \sqrt{1-t} \left|\frac{\partial^2 \phi}{\partial x_i \partial x_j} (\sqrt{1-t}x + \sqrt{t}b)\right| \leq \sqrt{1-t} \kder{\phi}{2} \\
		\intertext{and}
		&\left|\frac{\partial^2 \phi_i^{t,b}}{\partial x_j\partial x_k} (x)\right| = (1-t) \left|\frac{\partial^3 \phi}{\partial x_i \partial x_j \partial x_k} (\sqrt{1-t}x + \sqrt{t}b)\right| \leq (1-t) \kder{\phi}{3}
	\end{align}
	to deduce that
	\begin{align}
		\sup_{t \in (0,1)} |\psi'(t)| \leq& \tfrac{1}{2} \kder{\phi}{2} \sum_{i,j=1}^{m} \E \left[\left|C_{ij} - \langle \mal F_i, -\mal \opL F_j \rangle_{L^2(\lambda)}\right|\right] \notag\\
		&+ \sup_{t \in (0,1)} \frac{1}{2 \sqrt{1-t}} \left(2\sqrt{1-t}\kder{\phi}{2} \vee \tfrac{1}{2} (1-t) \kder{\phi}{3}\right)\notag\\
		&\hspace{5cm}\sum_{i=1}^{m} \int_\X \E  \left[\left(\sum_{j=1}^{m} |\mal_x F_j|\right)^q |\mal_x \opL F_i|\right] \lambda(dx).
	\end{align}
	The bound for the $d_3$ distance now follows.
	
	For the $d_2$ distance, as argued in the proof of Theorem~3.3 in \cite{PZ10}, it is enough to show that
	\begin{align}\label{Rainbow123}
		|\E[g(F) - g(X)]| \leq& A \norm{C^{-1}}_{op} \norm{C}_{op}^{1/2} \sum_{i,j=1}^m \E \left|C_{ij} - \langle \mal F_i, -\mal\opL F_j \rangle_{L^2(\lambda)} \right| \notag\\
		&+ \left(2A \norm{C^{-1}}_{op} \norm{C}_{op}^{1/2} \vee \frac{\sqrt{2\pi}}{8}B \norm{C^{-1}}_{op}^{3/2} \norm{C}_{op}\right)\notag\\
		&\hspace{4cm}\sum_{i=1}^{m} \int_\X \E  \left[\left(\sum_{j=1}^{m} |\mal_x F_j|\right)^q |\mal_x \opL F_i|\right] \lambda(dx),
	\end{align}
	for smooth functions $g \in \mathcal{C}^\infty(\R^m)$ whose first- and second-order derivatives are bounded in such a way that $\norm{g'}_\infty \leq A$ and $\sup_{x \in \R^m} \norm{\opHess g(x)}_{op} \leq B$. We now proceed as in the proof of Theorem~3.3 in \cite{PZ10} to deduce that
	\[
	\E g(F) - \E g(X) = \sum_{i,j=1}^{m} \E \left[C_{ij} \frac{\partial^2f_g}{\partial x_i \partial x_j} (F)\right] - \sum_{k=1}^{m} \E \left[\langle \mal \left(\frac{\partial f_g}{\partial x_k}(F)\right),-\mal \opL F_k \rangle_{L^2(\lambda)}\right],
	\]
	where $f_g$ is the canonical solution \eqref{eqcansol} to the multivariate Stein equation \eqref{eqMulStein}. Define $\phi_k(x) := \frac{\partial f_g}{\partial x_k}(x)$. By \Cref{lem3.1}, we have that
	\[
	\mal_x \phi_k(F) = \sum_{i=1}^{m} \frac{\partial \phi_k}{\partial x_i}(F) \cdot \mal_x F_i + R_{x,k},
	\]
	where
	\[
	|R_{x,k}| \leq \left(2 \norm{\phi_k'}_\infty \vee \tfrac{1}{2} \norm{\phi_k^{(2)}}_\infty\right) \left(\sum_{i=1}^{m} |\mal_x F_i|\right)^q.
	\]
	It follows that
	\begin{multline}
		|\E g(F) - \E g(X)| \leq \norm{f_g^{(2)}}_\infty \sum_{i,j=1}^m \E \left|C_{ij} - \langle \mal F_i, -\mal\opL F_j \rangle_{L^2(\lambda)} \right| \\
		+ \sup_{1\leq k \leq m}\left(2 \norm{\phi_k'}_\infty \vee \tfrac{1}{2} \norm{\phi_k^{(2)}}_\infty\right) \sum_{i=1}^{m} \int_\X \E  \left[\left(\sum_{j=1}^{m} |\mal_x F_j|\right)^q |\mal_x \opL F_i|\right] \lambda(dx).
	\end{multline}
	To see that \eqref{Rainbow123} holds, it suffices now to see that by \eqref{eqH1} and \eqref{eqH2}, we have
	\begin{align}
		&\kder{f_g}{2} \leq \norm{C^{-1}}_{op} \norm{C}_{op}^{1/2} \norm{g'}_\infty \leq A\norm{C^{-1}}_{op} \norm{C}_{op}^{1/2}\\
		&\norm{\phi_k'}_\infty \leq \kder{f_g}{2} \leq A\norm{C^{-1}}_{op} \norm{C}_{op}^{1/2}\\
		&\kder{\phi}{2} \leq \kder{f_g}{3} \leq \frac{\sqrt{2\pi}}{4} \norm{C^{-1}}_{op}^{3/2} \norm{C}_{op} \sup_{x \in \R^m} \norm{\opHess g(x)}_{op} \leq \frac{\sqrt{2\pi}}{4}B \norm{C^{-1}}_{op}^{3/2} \norm{C}_{op}.
	\end{align}
	This concludes the proof.
\end{proof}

\begin{proof}[Proof of \Cref{thmMulti}]
	Using \Cref{propMulti}, it suffices to show that
	\begin{equation}\label{eqFrodo1}
		\sum_{i,j=1}^{m} \E \left[\left|C_{ij} - \langle \mal F_i, -\mal \opL F_j \rangle_{L^2(\lambda)}\right|\right] \leq \zeta_1^{(p)} + \zeta_2^{(p)} + \zeta_3^{(p)}
	\end{equation}
	and
	\begin{equation}\label{eqFrodo2}
		\sum_{i=1}^{m} \int_\X \E  \left[\left(\sum_{j=1}^{m} |\mal_x F_j|\right)^q |\mal_x \opL F_i|\right] \lambda(dx) \leq \zeta_4^{(q)}.
	\end{equation}
	
	Fix $i,j\in \{1,...,m\}$. By the triangle inequality, we have that
	\begin{multline}\label{eqFrodo4}
		\E \left[|C_{ij} - \langle \mal F_i, -\mal \opL F_j \rangle_{L^2(\lambda)}|\right] \\\leq \E \left[|C_{ij} - \cov(F_i,F_j)|\right] + \E \left[|\cov(F_i,F_j) - \langle \mal F_i, -\mal \opL F_j \rangle_{L^2(\lambda)}|\right].
	\end{multline}
	
	Define $G_{ij}:= \cov(F_i,F_j) - \langle \mal F_i, -\mal \opL F_j \rangle_{L^2(\lambda)}$. Then since $F_i,F_j \in \dom \mal$, we have $G_{ij} \in L^1(\p_\eta)$ and by \Cref{lem:opL}, one has $\E G_{ij}=0$.
	
	Using the $p$-Poincaré inequality \eqref{eq:ppoin} given in \Cref{prop:ppoin}, we deduce that for any $p \in [1,2]$,
	\[
	\E |G_{ij}| \leq \E \left[|G_{ij}|^p\right]^{1/p} \leq 2^{2/p-1} \E \left[\int_\X |\mal_xG_{ij}|^p\lambda(dx)\right]^{1/p}.
	\]
	Now note that
	\[
	\mal_x G_{ij} = \mal_x \int_\X \left(\mal_y F_i\right) \cdot \left(-\mal_y \opL F_j\right) \lambda(dy),
	\]
	and by the argument in the proof of \cite[Prop.~4.1, p.~689]{LPS14}, we have that
	\[
	\left|\mal_x \int_\X \left(\mal_y F_i\right) \cdot \left(-\mal_y \opL F_j\right) \lambda(dy)\right| \leq  \int_\X \left|\mal_x\left(\left(\mal_y F_i\right) \cdot \left(-\mal_y \opL F_j\right)\right)\right| \lambda(dy).
	\]
	Using Minkowski's integral inequality, it follows that
	\begin{multline}\label{eqFrodo3}
		\E \left[|\cov(F_i,F_j) - \langle \mal F_i, -\mal \opL F_j \rangle_{L^2(\lambda)}|\right] \\\leq \left(\int_\X \left(\int_\X \E \left[\left|\mal_x\left(\left(\mal_y F_i\right) \cdot \left(-\mal_y \opL F_j\right)\right)\right|^p\right]^{1/p}\lambda(dy)\right)^p\lambda(dx)\right)^{1/p}.
	\end{multline}
	By \eqref{eq:ProdForm}, we have that
	\begin{multline}
		\mal_x\left(\left(\mal_y F_i\right) \cdot \left(-\mal_y \opL F_j\right)\right)\\ = \DD_{x,y}F_i\cdot \left(-\mal_y \opL F_j\right) + \mal_y F_i\cdot \left(-\DD_{x,y} \opL F_j\right) + \DD_{x,y}F_i\cdot \left(-\DD_{x,y} \opL F_j\right).
	\end{multline}
	Using Minkowski's norm inequality, the Cauchy-Schwarz inequality and \Cref{lem:opL}, one sees that
	\begin{align}
		\E \left[\left|\mal_x\left(\left(\mal_y F_i\right) \cdot \left(-\mal_y \opL F_j\right)\right)\right|^p\right]^{1/p}	\leq& \E \left[|\DD_{x,y}F_i|^{2p}\right]^{1/2p} \, \E \left[|\mal_{y}F_j|^{2p}\right]^{1/2p} \notag\\
		&+ \E\left[|\mal_{y}F_i|^{2p}\right]^{1/2p} \, \E \left[|\DD_{x,y}F_j|^{2p}\right]^{1/2p} \notag\\
		&+ \E\left[|\DD_{x,y}F_i|^{2p}\right]^{1/2p} \, \E \left[|\DD_{x,y}F_j|^{2p}\right]^{1/2p}. \label{eqFrodo5}
	\end{align}
	Combining \eqref{eqFrodo4}, \eqref{eqFrodo3} and \eqref{eqFrodo5} yields \eqref{eqFrodo1}.
	
	To show \eqref{eqFrodo2}, note that
	\[
	\left(\sum_{j=1}^{m} |\mal_x F_j|\right)^q \leq m^{q-1} \sum_{j=1}^{m} |\mal_x F_j|^q.
	\]
	The bound \eqref{eqFrodo2} now follows by Hölder inequality and \Cref{lem:opL}.
\end{proof}

\section{Proofs of \Cref{thmGilMulti1,thmGilMulti2}}\label{secPGilMulti}
Throughout this section, we denote by $c$ a positive absolute constant whose value can change from line to line. We will need some technical bounds derived in \cite{TT23} and presented in the next lemma.
\begin{lem}[{\cite[Prop.~F.2.]{TT23}}]\label{lem:GilBounds}
	Let $\alpha>-\frac{d}{2}$ and $r\geq 1$ such that $d+r\alpha>0$. Then $\lat(W) \in \dom \mal$ and there is a constant $c_0>0$ such that for all $x,y \in W$ and $t>0$,
	\begin{align}
		&\E \left[\left(\mal_x \lat\right)^r\right]^{1/r} \leq  c_0 \epsilon_t^\alpha (t\epsilon_t^d)^{1/r} (1 \vee t \epsilon_t^d)^{1-1/r} \label{eq:GilDxBound} \\
		\intertext{and}
		&\DD_{x,y}\lat = \ind{0<\norm{x-y}<\epsilon_t} \norm{x-y}^\alpha. \label{eq:GilDxy}
	\end{align}
\end{lem}

Moreover, we will need some properties of convex bodies. Let $W \subset \R^d$ be a convex body and define its inner parallel set $W_\epsilon$ by 
\begin{equation}\label{eq:Hdef}
W_\epsilon := \{x \in W : \text{dist}(x,\partial W)>\epsilon\},
\end{equation}
for $\epsilon>0$ and where $\text{dist}$ denotes the Euclidean distance and $\partial W$ the boundary of $W$. Since $W$ is a convex set with non-empty interior, the set $W_\epsilon$ is non-empty for $\epsilon>0$ small enough. Combining \cite[equation (3.19)]{HLS16} with Steiner's formula (cf. \cite[equation (14.5)]{StochIntGeo}), one sees that there is a constant $\gamma_W>0$ such that
\begin{equation}\label{eq:Hcond}
	|W \setminus W_\epsilon| \leq \gamma_W \epsilon.
\end{equation}
We can now start with the proof of \Cref{thmGilMulti1}.

\begin{proof}[Proof of \Cref{thmGilMulti1}]
	It suffices to bound the terms $\zeta_1^{(p)},...,\zeta_4^{(q)}$ of \Cref{thmMulti}. We start by giving a bound for $\zeta_1^{(p)}$. Note that is has been shown in \cite[Thm.~3.3]{TSRGilbert} that
	\[
	|W| - \gamma_W \epsilon_t < \frac{\cov\left(\lati{i},\lati{j}\right)}{\sigma_{ij}^{(1)}t^2\epsilon_t^{\alpha_i+\alpha_j+d} + \sigma_{ij}^{(2)}t^3 \epsilon_t^{\alpha_i+\alpha_j+2d}} \leq |W|,
	\]
	where $\gamma_W$ is a constant depending on $W$ such that \eqref{eq:Hcond} holds.
	This implies that
	\[
	\beta_{ij}^{(t)} \left(|W|-\gamma_W\epsilon_t\right) < \cov\left(\tillati{i},\tillati{j}\right) \leq \beta_{ij}^{(t)} |W|,
	\]
	where
	\[
	\beta_{ij}^{(t)} = \frac{\sigma_{ij}^{(1)}t^2\epsilon_t^{\alpha_i+\alpha_j+d} + \sigma_{ij}^{(2)}t^3 \epsilon_t^{\alpha_i+\alpha_j+2d}}{\left(t\epsilon_t^{\alpha_i+d/2} \vee t^{3/2}\epsilon_t^{\alpha_i+d}\right)\left(t\epsilon_t^{\alpha_j+d/2} \vee t^{3/2}\epsilon_t^{\alpha_j+d}\right)} = \frac{\sigma_{ij}^{(1)} + \sigma_{ij}^{(1)} t\epsilon_t^d}{1 \vee t\epsilon_t^d},
	\]
	as defined in \eqref{eq:betadef}. Hence we have
	\[
	\left|\cov\left(\tillati{i},\tillati{j}\right) - C_{ij}\right| \leq \beta_{ij}^{(t)}\gamma_W \epsilon_t + \left||W|\beta_{ij}^{(t)}-C_{ij}\right|
	\]
	and thus
	\[
	\zeta_1^{(p)} \leq \sum_{i,j=1}^{m} \left(\beta_{ij}^{(t)}\gamma_W \epsilon_t + \left||W|\beta_{ij}^{(t)}-C_{ij}\right|\right) \leq c \left(\epsilon_t + \max_{1 \leq i,j \leq m} \left|\beta_{ij}^{(t)} - c_{ij}\right|\right)
	\]
	for some constant $c>0$.
	
	Next, we bound $\zeta_2^{(p)}$. For this we use the expressions given in \Cref{lem:GilBounds}. We deduce from \eqref{eq:GilDxy} that
	\begin{multline}
		\zeta_2^{(p)} = 2^{2/p-1} \sum_{i,j=1}^m \left(t \epsilon_t^{\alpha_i+d/2} \vee t^{3/2} \epsilon_t^{\alpha_i+d}\right)^{-1} \left(t \epsilon_t^{\alpha_j+d/2} \vee t^{3/2} \epsilon_t^{\alpha_j+d}\right)^{-1} \\ \left(\int_W \left(\int_W \ind{|x-y|<\epsilon_t} |x-y|^{\alpha_i + \alpha_j}tdx\right)^{p}tdy\right)^{1/p}.
	\end{multline}
	Note that the inner integral is upper bounded by
	\[
	\frac{d\kappa_d}{d+\alpha_i+\alpha_j} t\epsilon_t^{d+\alpha_i+\alpha_j},
	\]
	and hence we deduce, after simplification,
	\[
	\zeta_2^{(p)} \leq 2^{2/p-1} |W|^{1/p} \left(\sum_{i,j=1}^m \frac{d\kappa_d}{d+\alpha_i+\alpha_j}\right) t^{-1+1/p} \left(1 \vee t\epsilon_t^d\right)^{-1}.
	\]
	For $\zeta_3^{(p)}$, after plugging in the bounds from \Cref{lem:GilBounds}, we get
	\begin{multline}
		\zeta_3^{(p)} \leq 2^{2/p}c \sum_{i,j=1}^m \left(t \epsilon_t^{\alpha_i+d/2} \vee t^{3/2} \epsilon_t^{\alpha_i+d}\right)^{-1} \left(t \epsilon_t^{\alpha_j+d/2} \vee t^{3/2} \epsilon_t^{\alpha_j+d}\right)^{-1} \\ \left(\int_W \left(\int_W \epsilon_t^{\alpha_i} (t\epsilon_t^d)^{1/(2p)}(1 \vee t\epsilon_t^d)^{1-1/(2p)} \ind{|x-y|<\epsilon_t} |x-y|^{\alpha_j}tdx\right)^{p}tdy\right)^{1/p}.
	\end{multline}
	After simplification, this bound yields
	\[
	\zeta_3^{(p)} \leq 2^{2/p}c|W|^{1/p} m\left(\sum_{i=1}^m \frac{d\kappa_d}{d+\alpha_i}\right) t^{-1+1/p} \left(1 \wedge t\epsilon_t^d\right)^{1/(2p)}.
	\]
	For $\zeta_4^{(q)}$, we plug in the first bound from \Cref{lem:GilBounds} and deduce
	\begin{multline}
		\zeta_4^{(q)} \leq m^{q-1}c^2 \sum_{i,j=1}^m \left(t \epsilon_t^{\alpha_i+d/2} \vee t^{3/2} \epsilon_t^{\alpha_i+d}\right)^{-1} \left(t \epsilon_t^{\alpha_j+d/2} \vee t^{3/2} \epsilon_t^{\alpha_j+d}\right)^{-q} \\\int_W \epsilon_t^{\alpha_i} (t\epsilon_t^d)^{1/(q+1)} (1 \vee t\epsilon_t^d)^{1-1/(q+1)} \epsilon_t^{q\alpha_j} (t\epsilon_t^d)^{q/(q+1)} (1 \vee t\epsilon_t^d)^{q-q/(q+1)} tdx,
	\end{multline}
	which, after simplification, yields
	\[
	\zeta_4^{(q)} \leq m^q c^2 |W| t^{(1-q)/2} (1 \wedge t\epsilon_t^d)^{(1-q)/2}.
	\]
	If we take $q=3-\frac{2}{p}$, then $q \in (1,2]$ and $(q+1)\alpha_i + d>2p\alpha_i+d>0$ and we get
	\[
	\zeta_2^{(p)} + \zeta_3^{(p)} + \zeta_4^{(q)} \leq c \left(t^{-1+1/p} \vee (t^2\epsilon_t^d)^{-1+1/p}\right),
	\]
	which concludes the proof.
\end{proof}
\begin{proof}[Proof of \Cref{thmGilMulti2}]
	As a first step, we compute the asymptotic covariance matrix of the vector $\tilde{F}_t$. Define the functions
	\[
	h_i(x,y) := \ind{|x-y|<\epsilon_t}\ind{x \in W_i}\ind{y \in W_i} |x-y|^{\alpha}, \quad i=1,\ldots,m.
	\]
	Then it holds that
	\[
	F_t^{(i)} = \frac{1}{2} \iint_{(\R^d)^2} h_i(x,y) (\eta^t)^{(2)}(dx,dy),
	\]
	where $(\eta^t)^{(2)}$ denotes the product measure $\eta^t \otimes \eta^t$. We deduce that for $i,j \in \{1,\ldots,m\}$,
	\begin{multline}\label{someRainbow}
		\cov\left(F_t^{(i)},F_t^{(j)}\right) = \frac{1}{4} \E \left[\iint_{(\R^d)^2} h_i(x,y) (\eta^t)^{(2)}(dx,dy) \iint_{(\R^d)^2} h_j(z,w) (\eta^t)^{(2)}(dz,dw)\right] \\ - \frac{1}{4} \E\left[\iint_{(\R^d)^2} h_i(x,y) (\eta^t)^{(2)}(dx,dy)\right]\E\left[\iint_{(\R^d)^2} h_j(z,w) (\eta^t)^{(2)}(dz,dw)\right].
	\end{multline}
	Since $\eta^t$ is a point measure, in the first term on the RHS of \eqref{someRainbow}, it is possible to have $x=z$ or $x=w$ or similar equalities, which constitute the diagonals of the sets we are summing over. Using the Mecke formula \eqref{eq:Mecke} and isolating these diagonals, one sees that 
	\begin{align}
		\cov\left(F_t^{(i)},F_t^{(j)}\right) &= \frac{1}{4}\iint_{(\R^d)^2} h_i(x,y) t^2\, dxdy \cdot \iint_{(\R^d)^2} h_j(z,w) t^2\, dzdw \notag\\
		&\phantom{=}+ \iiint_{(\R^d)^3} h_i(x,y)h_j(y,w) t^3\, dxdydw \notag\\
		&\phantom{=}+ \frac{1}{2}\iint_{(\R^d)^2} h_i(x,y)h_j(x,y) t^2\, dxdy\notag\\
		&\phantom{=}-\frac{1}{4}\iint_{(\R^d)^2} h_i(x,y) t^2\, dxdy \cdot \iint_{(\R^d)^2} h_j(z,w) t^2\, dzdw.
	\end{align}
	The first and the last term cancel, thus we are left with
	\begin{equation}\label{iP11cov}
		\cov\left(F_t^{(i)},F_t^{(j)}\right) = \iiint_{(\R^d)^3} h_i(x,y)h_j(y,w) t^3\, dxdydw + \frac{1}{2}\iint_{(\R^d)^2} h_i(x,y)h_j(x,y) t^2\, dxdy.
	\end{equation}
	We start by computing the first term on the RHS of \eqref{iP11cov}. We have
	\begin{multline}\label{iP11eq1}
		\iiint_{(\R^d)^3} h_i(x,y)h_j(y,w) t^3\, dxdydw \\= t^3 \int_{W_i \cap W_j}dy\int_{W_i}dx\int_{W_j}dw \ind{x \in B(y,\epsilon_t)} \ind{w \in B(y,\epsilon_t)} |x-y|^{\alpha} |w-y|^\alpha.
	\end{multline}
	Recall from \eqref{eq:Hdef} the definition of the (possibly empty) inner parallel set
	\[
	W_\epsilon := \{z \in W : \dist(x,\partial W)>\epsilon\},
	\]
	for $\epsilon>0$. Recall also from \eqref{eq:Hcond} that there is a constant $\gamma_W>0$ such that
	\begin{equation}\label{iP11Wcond}
		|W \setminus W_\epsilon| \leq \gamma_W \epsilon.
	\end{equation}
	We can now rewrite \eqref{iP11eq1} as 
	\begin{equation}\label{iP11eq2}
		t^3 \iiint_{(\R^d)^3} h_i(x,y)h_j(y,w)\, dxdydw = t^3 \int_{(W_i \cap W_j)_{\epsilon_t}}dy \left(\int_{B(y,\epsilon_t)}dx\, |x-y|^\alpha\right)^2 + R_t,
	\end{equation}
	where $R_t$ is given by
	\[
	R_t := t^3 \int_{(W_i \cap W_j)\setminus (W_i \cap W_j)_{\epsilon_t}}dy\int_{W_i}dx\int_{W_j}dw \ind{x \in B(y,\epsilon_t)} \ind{w \in B(y,\epsilon_t)} |x-y|^{\alpha} |w-y|^\alpha.
	\]
	The first term on the RHS of \eqref{iP11eq2} is given by
	\begin{equation}\label{iP11eq3}
		t^3\int_{(W_i \cap W_j)_{\epsilon_t}}dy \left(\int_{B(y,\epsilon_t)}dx\, |x-y|^\alpha\right)^2 = \left|(W_i \cap W_j)_{\epsilon_t}\right| \left(\frac{d\kappa_d}{d+\alpha}\right)^2 t^3\epsilon_t^{2d+2\alpha}.
	\end{equation}
	Using \eqref{iP11Wcond}, one sees that
	\begin{multline}\label{iP11eq4}
		0 \leq R_t \leq t^3\left|(W_i \cap W_j)\setminus (W_i \cap W_j)_{\epsilon_t}\right|\left(\int_{B(y,\epsilon_t)}dx\, |x-y|^\alpha\right)^2 \\\leq \gamma_{W_i \cap W_j}t^3\epsilon_t^{1+2d+2\alpha} \left(\frac{d\kappa_d}{d+\alpha}\right)^2.
	\end{multline}
	Combining \eqref{iP11eq2}, \eqref{iP11eq3} and \eqref{iP11eq4} with \eqref{iP11Wcond}, one sees that
	\begin{multline}\label{iP11eq5}
		\left|t^3\iiint_{(\R^d)^3} h_i(x,y)h_j(y,w) \, dxdydw - \left|W_i \cap W_j\right| \left(\frac{d\kappa_d}{d+\alpha}\right)^2 t^3\epsilon_t^{2d+2\alpha}\right| \\\leq \gamma_{W_i \cap W_j}t^3\epsilon_t^{1+2d+2\alpha} \left(\frac{d\kappa_d}{d+\alpha}\right)^2.
	\end{multline}
	For the second term in \eqref{iP11cov}, we proceed similarly. We have
	\begin{align}
		\tfrac{1}{2}t^2\iint_{(\R^d)^2} h_i(x,y)h_j(x,y)\, dxdy &= \tfrac{1}{2}t^2\iint_{W_i\cap W_j} \ind{|x-y|<\epsilon_t} |x-y|^{2\alpha} dxdy\\
		&=\tfrac{1}{2}t^2\int_{(W_i \cap W_j)_{\epsilon_t}}dy\int_{B(y,\epsilon_t)}dx |x-y|^{2\alpha} + R_t',
	\end{align}
	where
	\[
	0 \leq R_t' \leq \tfrac{1}{2}t^2\gamma_{W_i \cap W_j} \frac{d\kappa_d}{d+2\alpha} \epsilon_t^{d+2\alpha+1}.
	\]
	Hence we get
	\begin{multline}\label{iP11eq6}
		\left|\tfrac{1}{2}t^2\iint_{(\R^d)^2} h_i(x,y)h_j(x,y)\, dxdy - \tfrac{1}{2}\left|W_i \cap W_j\right| \frac{d\kappa_d}{d+2\alpha}t^2\epsilon^{d+2\alpha}\right| \\\leq \tfrac{1}{2} \gamma_{W_i \cap W_j} \frac{d\kappa_d}{d+2\alpha}t^2\epsilon^{d+2\alpha+1}.
	\end{multline}
	From \eqref{iP11cov}, \eqref{iP11eq5} and \eqref{iP11eq6}, we deduce that
	\begin{equation}\label{iP11eq7}
		\left|\frac{\cov\left(F_t^{(i)},F_t^{(j)}\right)}{\tfrac{1}{2} \frac{d\kappa_d}{d+2\alpha} t^2\epsilon_t^{d+2\alpha} + \left(\frac{d\kappa_d}{d+\alpha}\right)^2 t^3 \epsilon_t^{2d+2\alpha}} - \left|W_i \cap W_j\right|\right| \leq \gamma_{W_i \cap W_j}\epsilon_t.
	\end{equation}
	Now we use \Cref{thmMulti} and provide bounds for the terms $\zeta_1^{(p)},\ldots,\zeta_4^{(q)}$. 
	For the term $\zeta_1^{(p)}$, we have by \eqref{iP11eq7} that
	\[
	\zeta_1^{(p)} \leq \beta_{W_i \cap W_j} \epsilon_t.
	\]
	Plugging in the bounds from \Cref{lem:GilBounds}, we get for $\zeta_2^{(p)}$ that
	\begin{multline}
		\zeta_2^{(p)} \leq 2^{2/p-1}c t^{-2} \epsilon_t^{-2\alpha-d} \left(1 \vee t\epsilon_t^d\right)^{-1} \\\sum_{i,j=1}^{m} \left(\int_{W_i\cap W_j} \left(\int_{W_i \cap W_j} \ind{|x-y|<\epsilon_t} |x-y|^{2\alpha} tdx \right)^p tdy\right)^{1/p}.
	\end{multline}
	Simplifying, we deduce
	\[
	\zeta_2^{(p)} \leq c t^{-1+1/p} (1 \vee t\epsilon_t^d)^{-1} \sigma_1 \sum_{i,j=1}^{m} |W_i \cap W_j|^{1/p}.
	\]
	We proceed in the same way for $\zeta_3^{(p)}$ and $\zeta_4^{(q)}$, deducing
	\[
	\zeta_3^{(p)} \leq 2^{2/p}c m t^{-1+1/p} (1 \wedge t\epsilon_t^d)^{1/(2p)} \frac{d\kappa_d}{d+\alpha}\sum_{i=1}^{m} |W_i|^{1/p}
	\]
	and
	\[
	\zeta_4^{(q)} \leq m^{q-1}c t^{(1-q)/2} (1 \wedge t\epsilon_t^d)^{(1-q)/2} \sum_{i,j=1}^{m} |W_i \cap W_j|.
	\]
	If we take, as in the proof of \Cref{thmGilMulti1}, $q=3-\frac{2}{p}$, we get
	\[
	\zeta_2^{(p)} + \zeta_3^{(p)} + \zeta_4^{(q)} \leq c \left(t^{-1+1/p} \vee (t^2\epsilon_t^d)^{-1+1/p}\right),
	\]
	which concludes the proof.
\end{proof}

\appendix
\section{Some results on Poisson functionals}\label{secApp}
In this section, we collect some necessary results on Poisson functionals and Malliavin calculus. Let $\eta$ be a Poisson measure of intensity $\lambda$ on a $\sigma$-finite measure space $(\X,\lambda)$. We denote by $L^p(\p_\eta)$ the set of measurable functionals $F$ of $\eta$ such that $\E |F|^p < \infty$.

\textbf{Add-one cost.} Let $F$ be a measurable Poisson functional. For $x \in \X$, we define the add-one cost operator $\mal$ by
\[
\mal_x F := F(\eta + \delta_x) - F(\eta),
\]
fo $x \in \X$. For $n \geq 1$, we set inductively $\mal^{(n)} := \mal \mal^{(n-1)}$, where $\mal^{(0)}$ is the identity operator and $\mal^{(1)}=\mal$. We say that $F \in \dom \mal$ if $F \in L^2(\p_\eta)$ and
\[
\int_\X \E (\mal_x F)^2 \lambda(dx) < \infty.
\]
Note that the following product formula holds for $F,G$ measurable functionals of $\eta$:
\begin{equation}\label{eq:ProdForm}
\mal_x (FG) = \mal_xF \cdot G + F \cdot \mal_x G + \mal_xF \cdot \mal_xG.
\end{equation}

\textbf{Chaotic decomposition.} Let $n \in \N$ and $h\in L^2(\X^n,\lambda^{(n)})$. Denote by $I_n(h)$ the $n^{\text{th}}$ Wiener-Itô integral (see \cite[equation 25, p. 8]{LastPoiss}). Then for any $F \in L^2(\p_\eta)$, it holds that
\begin{equation} \label{eq:Chaos}
F = \sum_{n=0}^{\infty} I(f_n),
\end{equation}
where $f_n(x_1,\ldots,x_n) = \frac{1}{n!} \E \mal_{x_1,\ldots,x_n}^{(n)} F$ and $f_0 = I(f_0) = \E F$ and the series converges in $L^2(\p_\eta)$ (see \cite[Thm.~2]{LastPoiss}).

\textbf{Mecke formula.} Let $h=h(x,\eta)$ be a measurable, non-negative function of $\eta$ and $x \in \X$ such that 
\[
\E \int_\X |h(x,\eta)| \lambda(dx) <\infty.
\]

Then it holds that
\begin{equation} \label{eq:Mecke}
	\E \int_\X h(x,\eta) \eta(dx) = \E \int_\X h(x,\eta + \delta_x) \lambda(dx).
\end{equation}

\textbf{The operator $\opL$.} For functionals $F \in L^2(\p_\eta)$ having expansion \eqref{eq:Chaos}, we define the (pseudo) inverse of the Ornstein-Uhlenbeck generator $\opL$ by
\begin{equation}\label{eq:DefopL}
\opL F := - \sum_{n=1}^{\infty} \frac{1}{n} I_n(f_n).
\end{equation}
See also \cite[p.~24]{LastPoiss} for further details. We also use the following properties of the operator $\opL$, which can be found in \cite[Lemma~3.4 \& proof of Prop.~4.1]{LastPoiss}.
\begin{lem}\label{lem:opL}
	For any $F \in L^2(\p_\eta)$ and $r \geq 1$, it holds that
	\begin{align*}
		& \E |\mal_x\opL F |^r \leq \E |\mal_x F|^r, \quad \text{for } \lambda-\text{a.e. } x \in \X,\\
		\intertext{and}
		&\E |\DD_{x,y}\opL F |^r \leq \E |\DD_{x,y} F|^r, \quad \text{for } \lambda^{(2)}-\text{a.e. } (x,y) \in \X^2.
	\end{align*}
	Moreover, for $F,G \in \dom \mal$ with $\E F = \E G =0$, we have
	\[
	\cov(F,G) = \E \int_\X (\mal_x F) \cdot (-\mal_x \opL F) \lambda(dx).
	\]
\end{lem}

The improvement to the moment conditions given by the bounds of Theorem~\ref{thmMulti} comes from the following inequality, the so-called \textbf{$p$-Poincaré inequality}.

\begin{prop}[{\cite[(4.7) in Remark~4.4]{TT23}}]\label{prop:ppoin}
	Let $F \in L^1(\p_\eta)$ and $p \in [1,2]$. Then
	\begin{equation}\label{eq:ppoin}
	\E |F|^p - |\E F|^p \leq 2^{2-p} \int_\X \E |\mal_x F|^p \lambda(dx).
	\end{equation}
\end{prop}
When $p=2$, this inequality reduces to the classical Poincaré inequality, see e.g. \cite[Thm.~10]{LastPoiss}.

\bibliographystyle{alpha}
\bibliography{bibliography.bib}

\newcommand{\noop}[1]{}
\begin{thebibliography}{BOPT20}

\bibitem[BOPT20]{BOPT}
Andreas Basse-O’Connor, Mark Podolskij, and Christoph Thäle.
\newblock A {Berry}–{Esseen} theorem for partial sums of functionals of
  heavy-tailed moving averages.
\newblock {\em Electronic Journal of Probability}, 25:1 -- 31, 2020.

\bibitem[Gil61]{Gil61}
E.~N. Gilbert.
\newblock Random plane networks.
\newblock {\em Journal of the Society for Industrial and Applied Mathematics},
  9(4):533--543, 1961.

\bibitem[HLS16]{HLS16}
Daniel Hug, Günter Last, and Matthias Schulte.
\newblock Second-order properties and central limit theorems for geometric
  functionals of boolean models.
\newblock {\em The Annals of Applied Probability}, 26(1):73--135, 2016.

\bibitem[Las16]{LastPoiss}
Günter Last.
\newblock Stochastic analysis for {Poisson} processes.
\newblock In Giovanni Peccati and Matthias Reitzner, editors, {\em Stochastic
  Analysis for {Poisson} Point Processes}, volume~7 of {\em Bocconi \& Springer
  Series}, pages 1--36. Springer, 2016.

\bibitem[LPS16]{LPS14}
G\"{u}nter Last, Giovanni Peccati, and Matthias Schulte.
\newblock Normal approximation on {Poisson} spaces: {Mehler}'s formula, second
  order {Poincar\'{e}} inequalities and stabilization.
\newblock {\em Probab. Theory Related Fields}, 165(3-4):667--723, 2016.

\bibitem[LRSY19]{LRSY19}
Rapha{\"e}l Lachièze-Rey, Matthias Schulte, and J.E. Yukich.
\newblock Normal approximation for stabilizing functionals.
\newblock {\em The Annals of Applied Probability}, 29(2):931--993, 2019.

\bibitem[NP09]{PN09}
Ivan Nourdin and Giovanni Peccati.
\newblock {Stein's} method on {Wiener} chaos.
\newblock {\em Probab. Theory Relat. Fields}, 145:75--118, 2009.

\bibitem[NPR09]{PNR2009}
Ivan Nourdin, Giovanni Peccati, and Gesine Reinert.
\newblock Second order {Poincaré} inequalities and {CLTs} on {Wiener} space.
\newblock {\em Journal of Functional Analysis}, 257(2):593--609, 2009.

\bibitem[PSTU10]{PeccTaqq}
G.~Peccati, J.~L. Solé, M.~S. Taqqu, and F.~Utzet.
\newblock {Stein}’s method and normal approximation of {Poisson} functionals.
\newblock {\em The Annals of Probability}, 38(2):443 -- 478, 2010.

\bibitem[PZ10]{PZ10}
Giovanni Peccati and Cengbo Zheng.
\newblock Multi-dimensional {Gaussian} fluctuations on the {Poisson} space.
\newblock {\em Electron. J. Probab.}, 15(48):1487--1527, 2010.

\bibitem[RST17]{TSRGilbert}
Matthias Reitzner, Matthias Schulte, and Christoph Thäle.
\newblock Limit theory for the {Gilbert} graph.
\newblock {\em Advances in Applied Mathematics}, 88:26--61, 2017.

\bibitem[SW08]{StochIntGeo}
Rolf Schneider and Wolfgang Weil.
\newblock {\em Stochastic and Integral Geometry}.
\newblock Probability and Its Applications. Springer Berlin, Heidelberg, 1
  edition, 2008.

\bibitem[SY19]{SY19}
Matthias Schulte and J.~E. Yukich.
\newblock Multivariate second order {Poincaré} inequalities for {Poisson}
  functionals.
\newblock {\em Electron. J. Probab.}, 24(130):1--42, 2019.

\bibitem[SY23]{SY21}
Matthias Schulte and J.~E. Yukich.
\newblock Rates of multivariate normal approximation for statistics in
  geometric probability.
\newblock {\em The Annals of Applied Probability}, 33(1):507 -- 548, 2023.

\bibitem[Tal03]{Tal03}
Michel Talagrand.
\newblock {\em Spin Glasses: {A} Challenge for Mathematicians}.
\newblock A Series of Modern Surveys in Mathematics. Springer Berlin,
  Heidelberg, 1 edition, 2003.

\bibitem[Tra22]{TT23}
Tara Trauthwein.
\newblock Quantitative {CLTs} on the {Poisson} space via {Skorohod} estimates
  and {$p$-Poincaré} inequalities.
\newblock {\em arXiv:2212.03782}, 2022.

\end{thebibliography}
\end{document}